\documentclass[10pt, twoside, reqno]{amsart}
\usepackage{xcolor, amstext, amsfonts, amssymb, amsbsy, latexsym}
\usepackage{enumerate}
\usepackage[T1]{fontenc}
\usepackage{xy, hhline}
\newtheorem{lemma}{Lemma}[section]
\newtheorem{theo}[lemma]{Theorem}
\newtheorem{prop}[lemma]{Proposition}

\theoremstyle{definition}

\newtheorem{remark}[lemma]{Remark}

\numberwithin{equation}{section}

\newenvironment{proof_of}[1]{\medskip\noindent{\it Proof #1}.}
{$\Box$ \bigskip}
\newenvironment{eq}{\begin{equation}}{\end{equation}}
\renewcommand{\Ref}[1]{(\ref{#1})}

\newcommand{\Char}{\mathop{\rm char}}
\newcommand{\St}{\mathop{\rm St}}

\newcommand{\FF}{\mathbb{F}}

\newcommand{\CC}{\mathbb{C}}

\newcommand{\algA}{\mathcal{A}}

\newcommand{\K}[1]{K_{#1}}  

\newcommand{\scal}{\mathrm{scal}}     

\newcommand{\al}{\alpha}
\newcommand{\be}{\beta}

\newcommand{\la}{\lambda}
\newcommand{\de}{\delta}

\newcommand{\ov}[1]{\overline{#1}}
\newcommand{\un}[1]{{\underline{#1}} }

\newcommand{\tr}{\mathop{\rm tr}}
\newcommand{\Aut}{\mathop{\rm Aut}}
\newcommand{\diag}{\mathop{\rm diag}}

\newcommand{\G}{{\rm G}_2}
\newcommand{\SL}{{\rm SL}}
\newcommand{\GL}{{\rm GL}}

\newcommand{\matr}[4]{\left(\begin{array}{cc}
#1 & #2 \\
#3 & #4 \\
\end{array}\right)}

\newcommand{\OO}{\mathbf{O}}
\newcommand{\uu}{\mathbf{u}}
\newcommand{\vv}{\mathbf{v}}
\newcommand{\cc}{\mathbf{c}}
\newcommand{\zero}{\mathbf{0}}

\begin{document}
\title[Classification of $\G$-orbits for pairs of octonions]{Classification of $\G$-orbits for pairs of octonions}

\thanks{The work was supported by UAEU grant G00004229 and partially supported in accordance with the state task of the IM SB RAS, project FWNF-2022-0003. The first author was also supported by FAPESP 2023/17918-2.
}

\author{Artem Lopatin}
\address{Artem Lopatin\\
Universidade Estadual de Campinas (UNICAMP), 651 Sergio Buarque de Holanda, 13083-859 Campinas, SP, Brazil}
\email{dr.artem.lopatin@gmail.com (Artem Lopatin)}

\author{Alexandr N. Zubkov}
\address{Alexandr N. Zubkov\\
Department of Mathematical Sciences,
United Arab Emirates University,
Al Ain, Abu Dhabi, United Arab Emirates; 
Sobolev Institute of Mathematics, Omsk Branch, Pevtzova 13, 644043 Omsk, Russia}
\email{a.zubkov@yahoo.com (Alexandr N. Zubkov)}

\begin{abstract}
Over an algebraically closed field, we described a minimal set of representatives for $\G$-orbits on the set $\OO^2$ of pairs of octonions.


\end{abstract}

\maketitle

\section{Introduction}

\subsection{Results}\label{section_new} 

Assume that $\FF$ is an algebraically closed field of arbitrary characteristic $p=\Char\FF\geq0$. All vector spaces and algebras are over $\FF$. 

It is well-known that the problem of classification of all $\GL_n$-orbits on pairs of $n\times n$ matrices with respect to simultaneous conjugation is ``wild'', i.e., this classification problem for matrix pairs contains the classification problem for arbitrary systems of linear mappings on vector spaces~\cite{Belitskii_Sergeichuk_2003, Gelfand_Ponomarev_1969}. Algorithms for the solution of this problem were established by Friedland~\cite{Freidland_1983} and by  Belitskii~\cite{Belitskii_1983} (see also~\cite{Belitskii_2000, Sergeichuk_2000}).  Nevertheless, the explicit description of minimal set of representatives of $\GL_n$-orbits on pairs of matrices to the best of our knowledge is not known for $n>4$ (see Section 7.3 of~\cite{Belitskii_2000} for the case of $n\leq4$).

We consider the analogue of the above problem for pairs of octonions. Namely, the octonion algebra over $\FF$ is a simple alternative algebra of dimension 8. Its group of automorphisms $\G$ is a simple exceptional algebraic group. We described a minimal set of representatives for $\G$-orbits on the set $\OO^2$ of pairs of octonions (see Theorem~\ref{theo_prop_Can_min}) with respect to the diagonal action.  Our results were applied in~\cite{LZ_2} to obtain a separating set for polynomial $\G$-invariants of several octonions over an algebraically closed field of characteristic two, when a generating set for invariants is not known. Note that there is a connection between description of $\G$-orbits on $\OO^n$ and the investigation of the subgroup structure of $\G$ (see~\cite{Aschbacher_1987}). Our results can also be applied to the problem of solving polynomial equations over octonions. This problem has recently attracted a lot of attention (see for details~\cite{Chapman_2020_IJAC, Chapman_2020_JAA, Zhilina_Chapman_Guperman_Vishkautsan_2022, Chapman_Vishkautsan_2022, Chapman_Levin_2023, Flaut_Shpakivskyi_2014, Wang_Zhang_2014}).

In Sections~\ref{section_O} and \ref{section_G2} we explicitly define the octonion algebra $\OO$, its group of automorphisms $\G$  and the algebra of $\G$-invariants $\FF[\OO^n]^{\G}$ of $n$ copies of the algebra of octonions $\OO$. Additional properties of octonions  together with key notations of sets of octonions are given in Section~\ref{section_def}. The case of one octonion is considered in Section~\ref{section_one}, where we also present some techniques of dealing with orbits. The main result is proven in Section~\ref{section_two}.


\subsection{Octonions}\label{section_O}

The {\it octonion algebra} $\OO=\OO(\FF)$, also known as the {\it split Cayley algebra}, is the vector space of all matrices

$$a=\matr{\al}{\uu}{\vv}{\be}\text{ with }\al,\be\in\FF \text{ and } \uu,\vv\in\FF^3,$$%
endowed with the following multiplication:
$$a a'  =
\matr{\al\al'+ \uu\cdot \vv'}{\al \uu' + \be'\uu - \vv\times \vv'}{\al'\vv +\be\vv' + \uu\times \uu'}{\be\be' + \vv\cdot\uu'},\text{ where } a'=\matr{\al'}{\uu'}{\vv'}{\be'},$$%
$\uu\cdot \vv = u_1v_1 + u_2v_2 + u_3v_3$ and $\uu\times \vv = (u_2v_3-u_3v_2, u_3v_1-u_1v_3, u_1v_2 - u_2v_1)$. For short, denote  $\cc_1=(1,0,0)$, $\cc_2=(0,1,0)$,  $\cc_3=(0,0,1)$, $\zero=(0,0,0)$ from $\FF^3$. Consider the following basis of $\OO$: $$e_1=\matr{1}{\zero}{\zero}{0},\; e_2=\matr{0}{\zero}{\zero}{1},\; \uu_i=\matr{0}{\cc_i}{\zero}{0},\;\vv_i=\matr{0}{\zero}{\cc_i}{0}$$
for $i=1,2,3$. The unity of $\OO$ is denoted by $1_{\OO}=e_1+e_2$. We identify octonions
$$\al 1_{\OO},\;\matr{0}{\uu}{\zero}{0},\; \matr{0}{\zero}{\vv}{0}$$
with $\al\in\FF$, $\uu,\vv\in\FF^3$, respectively. Note that $\uu_i \uu_j= (-1)^{\epsilon_{ij}}\vv_k$ and  $\vv_i \vv_j= (-1)^{\epsilon_{ji}}\uu_k$, where $\{i,j,k\}=\{1,2,3\}$ and $\epsilon_{ij}$ is the parity of permutation 
$\left(
\begin{array}{ccc}
\!1\! & \!2\! & \!3\! \\
\!k\! & \!i\! & \!j\! \\
\end{array}
\right)$.

Similarly to $\OO(\FF)$ we define the algebra of octonions $\OO(\algA)$ over any commutative associative $\FF$-algebra $\algA$.

The algebra $\OO$ has a linear involution
$$\ov{a}=\matr{\be}{-\uu}{-\vv}{\al},\text{ satisfying  }\ov{aa'}=\ov{a'}\ov{a},$$%
a {\it norm} $n(a)=a\ov{a}=\al\be-\uu\cdot \vv$, and a non-degenerate symmetric bilinear {\it form} $q(a,a')=n(a+a')-n(a)-n(a')=\al\be' + \al'\be - \uu\cdot \vv' - \uu'\cdot \vv$. Define the linear function {\it trace} by $\tr(a)=a + \ov{a} = \al+\be$. The subspace $\{a\in\OO\,|\,\tr(a)=0\}$ of traceless octonions is denoted by $\OO_0$. Notice that
\begin{eq}\label{eq1}
\tr(aa')=\tr(a'a) \text{ and } n(aa')=n(a)n(a').
\end{eq}%
The next quadratic equation holds:
\begin{eq}\label{eq2}
a^2 - \tr(a) a + n(a) = 0.
\end{eq}%
Since $n(a+a')=n(a)+n(a')-\tr(aa') + \tr(a)\tr(a')$, the linearization of equation~\Ref{eq2} implies
\begin{eq}\label{eq3}
aa' + a'a - \tr(a) a' - \tr(a') a  -\tr(aa') + \tr(a)\tr(a') = 0.
\end{eq}%
The algebra $\OO$ is a simple {\it alternative} algebra, i.e., the following identities hold for $a,b\in\OO$:
\begin{eq}\label{eq4}
a(ab)=(aa)b,\;\; (ba)a=b(aa).
\end{eq}%
The linearization implies that
\begin{eq}\label{eq5}
a(a'b) + a'(ab)=(aa'+a'a)b,\;\; (ba)a' + (ba')a = b(aa'+a'a).
\end{eq}%
The trace is associative, i.e., for all $a,b,c\in\OO$ we have
\begin{eq}\label{eq6}
\tr((ab)c) = \tr(a(bc)).
\end{eq}%
Note that
\begin{eq}\label{eq_n}
2n(a)=-\tr(a^2)+{\tr}^2(a) \text{ for each }a\in\OO.
\end{eq}%
More details on $\OO$ can be found in Sections 1 and 3 of~\cite{zubkov2018}.

\subsection{The group $\G$}\label{section_G2} The group $\G=\G(\FF)$ is known to be the group $\Aut(\OO)$ of all automorphisms of the algebra $\OO$. The group $\G$ contains  a Zariski closed subgroup $\SL_3=\SL_3(\FF)$. Namely, every $g\in\SL_3$ defines the following automorphism of $\OO$:
$$a\to \matr{\al}{\uu g}{\vv g^{-T}}{\be},$$
where $g^{-T}$ stands for $(g^{-1})^T$ and $\uu,\vv\in\FF^3$ are considered as row vectors. For every $\uu,\vv\in \OO$ define $\de_1(\uu),\de_2(\vv)$ from $\Aut(\OO)$ as follows:
$$\de_1(\uu)(a')=\matr{\al' - \uu\cdot \vv'}{(\al'-\be' - \uu\cdot \vv')\uu + \uu'}{\vv' - \uu'\times \uu}{\be' + \uu\cdot\vv'},$$
$$\de_2(\vv)(a')=\matr{\al' + \uu'\cdot \vv}{\uu' + \vv'\times \vv}{(-\al'+\be' - \uu'\cdot \vv)\vv + \vv'}{\be' - \uu'\cdot\vv}.$$
The group $\G$ is generated by  $\SL_3$ and $\de_1(t\uu_i),\de_2(t\vv_i)$ for all $t\in\FF$ and $i=1,2,3$. As an example, by straightforward calculations we can see that
\begin{eq}\label{eq_h}
\hbar:\OO\to\OO, \text{ defined by }a \to \matr{\be}{-\vv}{-\uu}{\al},
\end{eq}%
belongs to $\G$ (see also the proof of Lemma 1 of~\cite{zubkov2018}).

The action of $\G$ on $\OO$ satisfies the next properties:
$$\ov{ga}=g\ov{a},\; \tr(ga)=\tr(a),\;n(ga)=n(a),\; q(ga,ga')=q(a,a').$$%
Thus, $\G$ acts also on $\OO_0$. The group $\G$  acts diagonally on the vector space $\OO^n=\OO\oplus \cdots\oplus\OO$ ($n$ times) by $g(a_1,\ldots,a_n)=(ga_1,\ldots,g a_n)$
for all $g\in \G$ and $a_1,\ldots,a_n\in\OO$.

\subsection{Notations}
We write $E$ for the identity matrix. Denote by $E_{ij}$ the matrix such that the $(i,j)^{\rm th}$ entry is equal to one and the rest of entries are zeros. Given a matrix $A$, denote by $(A)_{ij}$ the $(i,j)^{\rm th}$ entry of $A$. We use the symbol $\ast$ to denote some element of $\FF$ or $\FF^3$ of an octonion (as an example, see Lemma~\ref{lemma_St_K1}). We fix a binary relation $<$ on the field $\FF$ such that for each pair $\al,\be\in \FF$ with $\al\neq\be$ exactly one of $\al<\be$ or $\be<
\al$ holds. Note that we do not assume that $<$ is transitive, and we do not assume compatibility with the field operations.

\section{Auxiliaries}\label{section_def}

\subsection{Polynomial $\G$-invariants}\label{section_pol_G2_inv}

The coordinate algebra of the affine variety $\OO^n$ is the polynomial $\FF$-algebra $\K{n}=\FF[\OO^n]=\FF[z_{ij}\,|\,1\leq i\leq n,\; 1\leq j\leq 8]$, where $z_{ij}:\OO^n\to\FF$ is defined by $(a_1,\ldots,a_n)\to \al_{ij}$ for 
\begin{eq}\label{eq7}
a_i=(\al_{ij}\,|\,1\leq j\leq 8)=\matr{\al_{i1}}{(\al_{i2}, \al_{i3}, \al_{i4})}{(\al_{i5}, \al_{i6}, \al_{i7})}{\al_{i8}}\in\OO.
\end{eq}%
The algebra of {\it polynomial $\G$-invariants of several octonions} is
$$\K{n}^{\G}=\{f\in \FF[\OO^n]\,|\,f(g\un{a})=f(\un{a}) \text{ for all }g\in\G,\; \un{a}\in \OO^n\}.$$%
Note that the trace $\tr$ and the norm $n$ belong to $\K{1}^{\G}$. More details about polynomial $\G$-invariants of several octonions can be found for example in~\cite{LZ_2}. A generating set for $\K{n}^{\G}$ was constructed by Schwarz~\cite{schwarz1988} over the field of complex numbers $\CC$. This result has been generalized to an arbitrary infinite field of odd characteristic by Zubkov and Shestakov in~\cite{zubkov2018}.

Assume that $S$ is a set of functions $\OO^n\to \FF$, which are constants on $\G$-orbits on $\OO^n$.  We say that $S$  {\it separates orbits} on $\OO^n$ if for every $\un{a},\un{b}\in\OO^n$ the condition $f(\un{a})=f(\un{b})$ for all $f\in S$ implies that $\G\un{a}=\G\un{b}$. Let us remark that separating polynomial invariants that were introduced by Derksen and Kemper in~\cite{derksen2002computational} (see~\cite{derksen2002computationalv2} for the second edition) and then were studied in~\cite{Cavalcante_Lopatin_1, DM5, domokos2017, Domokos20Add, Domokos20, dufresne2014,  Ferreira_Lopatin_2023, kaygorodov2018, Kemper_Lopatin_Reimers_2022, kohls2013, Lopatin_Reimers_1},  separate only closed orbits in Zariski topology.

\subsection{Sets of octonions}\label{section_types}

Introduce the following sets of diagonal octonions:
$$ {\rm(D)}\; \matr{\al_1}{\zero}{\zero}{\al_8},\qquad\quad 
{\rm(E)}\; \al_1 1_{\OO}, \qquad\quad
{\rm(F)}\; \matr{\al_1}{\zero}{\zero}{\al_8} \text{ with }\al_1\neq \al_8,
$$
where $\al_1,\al_8\in\FF$. Note that set (D) is the union of sets (E) and (F). Given $a\in\OO$, denote by $a^{\top}$ the transpose octonion of $a$, i.e.
$$a^{\top}=\matr{\al}{\vv}{\uu}{\be} \text{ for }a=\matr{\al}{\uu}{\vv}{\be}.$$
\noindent{}Consider the following sets of non-diagonal octonions:
\begin{enumerate}

\item[(K)] $\matr{\al_1}{(1,0,0)}{\zero}{\al_8}$,

\item[(L)] $\matr{\al_1}{(\al_2,0,0)}{\zero}{\al_8}$ with $\al_2\neq0$,

\item[(M)] $\matr{\al_1}{(0,1,0)}{\zero}{\al_8}$,

\item[(N)] $\matr{\al_1}{(1,0,0)}{(\al_5,0,0)}{\al_8}$  with $\al_5\neq0$,

\item[(${\mathrm{P}}$)] $\matr{\al_1}{(1,0,0)}{(0,1,0)}{\al_8}$,
\end{enumerate}
for $\al_1,\al_2,\al_5,\al_8\in\FF$.  Given some set (A) of octonions, we denote by
\begin{enumerate}
\item[($\mathrm{A}_0$)] the set of octonions $a\in\OO_0$ from set (A);

\item[($\mathrm{\ov{A}}$)] the set of octonions $a\in\OO$ from set (A) with $\al_1\leq\al_8$;

\item[($\mathrm{A}_1$)] the set of octonions $a\in\OO$ from set (A) with $\al_1=\al_8$;

\item[($\mathrm{A}^{\!\top}$)] the set of octonions $a^{\top}$ such that $a$ belongs to set (A);

\item[($\mathrm{A}_0^{\!\top}$)] the set of octonions $a^{\top}$ such that $a\in\OO_0$ belongs to set (A);

\item[($\mathrm{A}_1^{\!\top}$)] the set of octonions $a^{\top}$ such that $a$ belongs to set (A) and $\al_1=\al_8$.
\end{enumerate}
\noindent{}Note that in case $\Char{\FF}=2$ the sets ($\mathrm{A}_0$) and ($\mathrm{A}_1$) coincide.

\subsection{Stabilizers in $\G$}\label{section_properties_G2}

\begin{lemma}\label{lemma_St_D} Given $\al,\be\in \FF$ with $\al\neq\be$, the stabilizer 
$\St_{\G}(\al e_1 + \be e_2)$ is equal to $\SL_3$.
\end{lemma}
\begin{proof} The inclusion $\SL_3\subset \St_{\G}(\al e_1 + \be e_2)$ is trivial. 

To prove the reverse inclusion, we assume that $g a=a$ for some $g\in\G$, where $a$ stands for $\al e_1 + \be e_2$. We have
\begin{eq}\label{eq_ge}
g(e_1) = e_1 \;\text{ and }\; g(e_2) = e_2, 
\end{eq}%
\noindent{}since $g((\al-\be) e_1) = g(a - \be 1_{\OO}) = g(a) - \be 1_{\OO} = (\al - \be) e_1$ and $g((\be-\al) e_2) = g(a - \al 1_{\OO}) = g(a) - \al 1_{\OO} = (\be - \al) e_2$. Moreover,
\begin{eq}\label{eq_guv}
g(\uu_i)= \matr{0}{\ast}{\zero}{0} \;\text{ and }\;  g(\vv_i)= \matr{0}{\zero}{\ast}{0} \;\text{ for }\; 1\leq i\leq3,
\end{eq}%
\noindent{}since we act by $g$ on both sides of equalities $\uu_i e_1 = 0$, $e_1 \uu_i=\uu_i$, $\vv_i e_2 = 0$, $e_2 \vv_i= \vv_i$, and apply obvious formulas 
$$b e_1 = \matr{\be_1}{\zero}{\vv}{0}, e_1 b  = \matr{\be_1}{\uu}{0}{0} \;\text{ and }\;
b e_2 = \matr{0}{\uu}{0}{\be_8}, e_2 b  = \matr{0}{\zero}{\vv}{\be_8} $$%
for every $b=\matr{\be_1}{\uu}{\vv}{\be_8}$.
Formulas~\Ref{eq_guv} imply that the linear map $g:\OO\to\OO$ can be restricted to the $\FF$-span of $\uu_1,\uu_2,\uu_3$. Therefore, there exists $A\in\GL_3$ such that $g \matr{0}{\uu}{\zero}{0} = \matr{0}{\uu A}{\zero}{0}$
for all $\uu\in \FF^3$. Similarly, there exists $B\in\GL_3$ such that 
$g \matr{0}{\zero}{\vv}{0} = \matr{0}{\zero}{\vv B}{0}$
for all $\vv\in \FF^3$. Together with formula~\Ref{eq_ge} we obtain that 
\begin{eq}\label{eq_g}
g \matr{\al_1}{\uu}{\vv}{\al_8} = \matr{\al_1}{\uu A}{\vv B}{\al_8}. 
\end{eq}

For $\de=(\det(A))^{1/3}$ denote $A_0 = A / \de$. Then $h=A_0^{-1}\in \SL_3$ is an element of $\G$ and for $C=BA_0^T$ we obtain
$$hg \matr{\al_1}{\uu}{\vv}{\al_8} = \matr{\al_1}{\uu\, \de}{\vv C}{\al_8}. $$

We act by $hg$ on both sides of equality $\uu_i \vv_i= e_1$, where $1\leq i\leq 3$, and obtain
$$(\uu_i \de)(\vv_i C) = e_1$$
and $\de C_{ii}=1$. Therefore, $C_{ii}=1/\de$ for all $1\leq i\leq 3$.

We act by $hg$ on both sides of equality $\uu_i \uu_j= (-1)^{\epsilon_{ij}}\vv_k$, where $\{i,j,k\}=\{1,2,3\}$, and obtain
$$(\uu_i \de)(\uu_j \de) = (-1)^{\epsilon_{ij}} \matr{0}{\zero}{\vv_k C}{0}.$$
Thus $\de^2 \vv_k = \vv_k C$ in $\FF$-span of $\vv_1,\vv_2,\vv_3$. Considering coefficients of $\vv_i,\vv_j,\vv_k$, respectively, we obtain $C_{ki}=0$, $C_{kj}=0$, $C_{kk}=\de^2$. The last equality implies that $\de^3=1$, i.e., $A\in\SL_3$. Finally, we obtain that $BA^T = \de^3 E =E$. Therefore, $g=A\in\SL_3$.
\end{proof}

\begin{lemma}\label{lemma_St_K1} Assume $g\in\G$ and $\al,\la\in \FF$. Then 
\begin{enumerate}
\item[(a)] $\St_{\G}(\al 1_{\OO} + \uu_1) = \St_{\G}(\uu_1)$;

\item[(b)] if $g\in \St_{\G}(\al 1_{\OO} + \uu_1)$, then 
$$g(e_1)=\matr{1}{(\ast,0,0)}{(0,\ast,\ast)}{0} \;\text{ and }\; g(e_2)=\matr{0}{(\ast,0,0)}{(0,\ast,\ast)}{1};$$

\item[(c)] $\de_1(\la\uu_1) \in \St_{\G}(\al 1_{\OO} + \uu_1)$.
\end{enumerate}
\end{lemma} 
\begin{proof} For $a=\al 1_{\OO} + \uu_1$ we have 
$g(\uu_1) = g(a-\al 1_{\OO}) = g(a)- \al 1_{\OO}$. Therefore, part (a) is proven.

As in formula~\Ref{eq7}, consider some element $b=(\be_i\,|\,1\leq i\leq 8)$ from $\OO$. Then 
$$
b \uu_1  = \matr{0}{(\be_1,0,0)}{(0,\be_4,-\be_3)}{\be_5} \;\text{ and }\; 
\uu_1 b  = \matr{\be_5}{(\be_8,0,0)}{(0,-\be_4,\be_3)}{0}. 
$$
Acting by $g\in \St_{\G}(\al 1_{\OO} + \uu_1)$ on the equalities $e_1\uu_1=\uu_1$, $\uu_1 e_1=0$ and $e_2 \uu_1 = 0$, $\uu_1 e_2 = \uu_1$, and applying the above formulas, we conclude the proof or part (b).

The proof of part (c) is straightforward.
\end{proof}

\section{Classification of octonions}\label{section_one}

In this section we write $\al_1,\ldots,\al_8$ for some elements of $\FF$. 


\begin{remark}\label{rem1}
Assume $a=\matr{\al}{\uu}{\vv}{\be}\in\OO$, $\la\in\FF$ and $1\leq i,j\leq 3$, $i\neq j$. Then for $g=E+\la E_{ij}\in \SL_3$ we have that  $ga=\matr{\al}{\uu'}{\vv'}{\be}$, where
\begin{enumerate}
\item[$\bullet$] $u'_j= u_j + \la u_i$ and $u'_k=u_k$ for all $k\neq j$;

\item[$\bullet$] $v'_i= v_i - \la v_j$ and $v'_k=v_k$ for all $k\neq i$.
\end{enumerate}
\end{remark}
\medskip%

\begin{lemma}\label{lemma_new}

\begin{enumerate}
\item[(a)] Acting by $\SL_3$ on $\OO$ we can make the following reduction: 
$\matr{\al_1}{\uu}{\vv}{\al_8} \to  \matr{\al_1}{(1,0,0)}{\vv'}{\al_8}$, where $\vv'=(\ast,0,0)$ or $\vv'=(0,1,0)$, in case $\uu\neq 0$.

\smallskip
\noindent{} More specifically:

\item[(b)] $\matr{\al_{1}}{\zero}{\vv}{\al_{8}}\to \matr{\al_{1}}{\zero}{(1, 0, 0)}{\al_{8}}$ in case $\vv\neq\zero$;

\item[(c)] $\matr{\al_{1}}{(\al_2, 0, 0)}{(\al_{5}, \al_{6}, \al_{7})}{\al_{8}} \to \matr{\al_{1}}{(\al_2, 0, 0)}{(\al_5, 0, 0)}{\al_{8}}$ and,  simultaneously, $(\uu_1,\vv_1)\to (\uu_1,\vv_1)$, in case $\al_5\neq 0$;

\item[(d)] $\matr{\al_{1}}{(\al_2, 0, 0)}{(\al_{5}, \al_{6}, \al_{7})}{\al_{8}}\to \matr{\al_{1}}{(\al_2, 0, 0)}{(\al_5, 1, 0)}{\al_{8}}$ and, simultaneously, $(\uu_1,\vv_1)\to(\uu_1,\vv_1)$, in case $\al_6$ or $\al_7$ is non-zero;

\item[(e)] $\matr{\al_{1}}{(\al_2,0,0)}{(-1,0,0)}{\al_{8}} \to  \matr{\al_{1}}{(1,0,0)}{(-\al_2,0,0)}{\al_{8}}$, in case $\al_2\neq0$;

\item[(f)] $\matr{\al_{1}}{(0,-1,0)}{(-1,0,0)}{\al_{8}}\to\matr{\al_{1}}{(1,0,0)}{(0,1,0)}{\al_{8}}$.
\end{enumerate}
\end{lemma}
\begin{proof} Parts (a)---(d)  follow from Remark~\ref{rem1}.  To prove part (e) consider the action by the diagonal matrix $g=\diag(1/\al_2, \al_2, 1)$ from $\SL_3$. To prove part (f) consider the action by the following matrix from $\SL_3$: 
$$g=\left(
\begin{array}{ccc}
0  & -1 & 0 \\
-1 & 0 & 0 \\
0 & 0 & -1 \\
\end{array}
\right).$$ 
\end{proof}

Define the function $\scal:\OO\to \FF$ by
$$ \scal(a)= \left\{
\begin{array}{rl}
1, & \text{ if } a=\al 1_{\OO} \text{ for some }\al\in\FF \\
0, & \text{ otherwise }\\
\end{array}
\right.$$%
for every $a\in\OO$. Note that $\scal$ is constant on $\G$-orbits on $\OO^n$. The weaker version of following proposition is contained in the proof of Lemma~3.4 of~\cite{Elduque_2023}. We present the proof for the sake of completeness.

\begin{prop}\label{prop_Can3.1_new} \,

\begin{enumerate}
\item[1.] The following set is a minimal set of representatives of $\G$-orbits on $\OO$: 
\begin{enumerate}
\item[{\rm(E)}] $\al_1 1_{\OO}$,

\item[{\rm($\mathrm{\ov{F}}$)}] $\matr{\al_1}{\zero}{\zero}{\al_8}$ with $\al_1<\al_8$,

\item[{\rm($\mathrm{K}_1$)}] $\matr{\al_1}{(1,0,0)}{\zero}{\al_1}$,
\end{enumerate}
where $\al_1,\al_8\in\FF$. 

\item[2.] The polynomial invariants $\tr$, $n$ together with the function $\scal$ separate $\G$-orbits on $\OO$.
\end{enumerate}
\end{prop}
\begin{proof} Denote by $S$ the set from the formulation of this proposition, and consider some $a=\matr{\al_1}{\uu}{\vv}{\al_8}$ from $\OO$.

If $\uu=\vv=0$, then acting by $\hbar$ on $a$ we obtain the octonion from set ($\mathrm{\ov{F}}$). Otherwise, acting by $\hbar$ and applying part (a) of Lemma~\ref{lemma_new} we can assume that case one of the following two cases holds.

\medskip
\noindent{\bf (a)}
Assume $a=\matr{\al_1}{(1,0,0)}{(\al_5,0,0)}{\al_8}$. Since $\FF$ is algebraically closed, then there exists $v_1\in\FF$ such that $v_1^2 + v_1(\al_1-\al_8)-\al_5=0$. Acting by $\de_2(v_1,0,0)\in\G$ we can assume that $a=\matr{\al'_1}{(1,0,0)}{\zero}{\al'_8}$ for some $\al'_1,\al'_2\in\FF$. 

If $\al'_1=\al'_8$, then $a$ is an octonion from set ($\mathrm{K}_1$). Otherwise, we act by $\de_1(1/(\al'_8-\al'_1),0,0)$ on $a$ to obtain the octonion $\matr{\al'_1}{\zero}{\zero}{\al'_8}$; then acting by $\hbar$ we obtain the octonion from set ($\mathrm{\ov{F}}$).

\medskip
\noindent{\bf (b)}
Assume $a=\matr{\al_1}{(1,0,0)}{(0,1,0)}{\al_8}$. Acting by $\de_1(0,0,-1)$ on $a$ we obtain $\matr{\al_1}{(1,0,\al_8-\al_1)}{\zero}{\al_8}$. By part (a) of Lemma~\ref{lemma_new}  we obtain $\matr{\al_1}{(1,0,0)}{\zero}{\al_8}$. This case has already been considered in part (a).

Therefore, $S$ contains representatives of all $\G$-orbits on $\OO$. To prove the minimality of $S$ we assume the contrary, i.e., $ga=b$ for some $a,b\in S$ with $a\neq b$ and $g\in\G$. Since $\tr$ and $n$ are $\G$-invariants, without loss of generality can assume that $a=\al_1 1_{\OO}$ belongs to set (E) and $b=\al_1 1_{\OO} + \uu_1$ belongs to set ($\mathrm{K}_1$). Then $a$ and $b$ are separated by $\scal$; a contradiction to the condition $ga=b$.

Part 2 follows from the proof of part 1.
\end{proof}


\section{Classification of a pair of octonions}\label{section_two}

\begin{theo}\label{theo_prop_Can4.1}
For each $(a,b)\in\OO^2$ there exists $g\in\G$ such that $g (a,b)$ is a pair from one of the following sets: 
\begin{enumerate}
\item[{\rm (DD)}] $\left(\matr{\al_1}{\zero}{\zero}{\al_8}, \matr{\be_1}{\zero}{\zero}{\be_8}\right)$,

\item[{\rm($\rm E K_1$)}] $\left(\al_1 1_{\OO}, \matr{\be_1}{(1,0,0)}{\zero}{\be_1}\right)$,

\item[{\rm(FK)}] $\left(\matr{\al_1}{\zero}{\zero}{\al_8}, \matr{\be_1}{(1,0,0)}{\zero}{\be_8}\right)$ with $\al_1\neq \al_8$,

\item[{\rm(FN)}] $\left(\matr{\al_1}{\zero}{\zero}{\al_8}, \matr{\be_1}{(1,0,0)}{(\be_5,0,0)}{\be_8}\right)$ with $\al_1\neq \al_8$ and $\be_5\neq0$,

\item[{\rm(FP)}] $\left(\matr{\al_1}{\zero}{\zero}{\al_8}, \matr{\be_1}{(1,0,0)}{(0,1,0)}{\be_8}\right)$ with $\al_1\neq \al_8$,

\item[{\rm($\rm K_1E$)}] $\left(\matr{\al_1}{(1,0,0)}{\zero}{\al_1}, \be_1 1_{\OO}\right)$,

\item[{\rm($\rm K_1F$)}] $\left(\matr{\al_1}{(1,0,0)}{\zero}{\al_1}, \matr{\be_1}{\zero}{\zero}{\be_8}\right)$ with $\be_1\neq \be_8$,

\item[{\rm($\rm K_1L_1$)}] $\left(\matr{\al_1}{(1,0,0)}{\zero}{\al_1}, \matr{\be_1}{(\be_2,0,0)}{\zero}{\be_1}\right)$ with $\be_2\neq 0$,

\item[{\rm($\rm K_1L^{\!\top}$)}] $\left(\matr{\al_1}{(1,0,0)}{\zero}{\al_1}, \matr{\be_1}{\zero}{(\be_5,0,0)}{\be_8}\right)$ with $\be_5\neq 0$,

\item[{\rm($\rm K_1M$)}] $\left(\matr{\al_1}{(1,0,0)}{\zero}{\al_1}, \matr{\be_1}{(0,1,0)}{\zero}{\be_8}\right)$,

\item[{\rm($\rm K_1M_1^{\!\top}$)}] $\left(\matr{\al_1}{(1,0,0)}{\zero}{\al_1}, \matr{\be_1}{\zero}{(0,1,0)}{\be_1}\right)$,
\end{enumerate}
for all $\al_1,\al_8,\be_1,\be_2,\be_5,\be_8\in\FF$.
\end{theo}

To prove Theorem~\ref{theo_prop_Can4.1} we need the next two lemmas.

\begin{lemma}\label{lemma_Claim_4.3}
Assume $b\in\OO$. Then acting by $\SL_3<\G$ and $\hbar\in\G$ on $b$ we can assume that $b$ belongs to one of the following sets: {\rm (D)}, {\rm (K)}, {\rm (N)}, {\rm (P)}.
\end{lemma}
\begin{proof} Assume $b=\matr{\al}{\uu}{\vv}{\be}$ does not belong to set (D). Then $\uu\neq0$ or $\vv\neq0$. Acting by $\hbar$, we can assume that $\uu\neq0$. By part (a) of Lemma~\ref{lemma_new}, we can assume that $\uu=(1,0,0)$ and $\vv=(\ast,0,0)$ or $\vv=(0,1,0)$. The proof is concluded. 
\end{proof}

\begin{lemma}\label{lemma_Claim_4.4}
Assume $a,b\in\OO$ and $a$ belongs to set  {\rm($\mathrm{K_1}$)}. Then there exists $g\in \SL_3$ such that $g(a,b)=(a,b')$, where $b'$ is one of the next octonions:
\begin{enumerate}
\item[1.] $b'=\matr{\ast}{(\ast,0,0)}{(\ast,0,0)}{\ast}$, 

\item[2.] $b'=\matr{\ast}{(\ast,0,0)}{(0,1,0)}{\ast}$, 

\item[3.] $b'=\matr{\ast}{(0,1,0)}{(0,\ast,\ast)}{\ast}$, 

\item[4.] $b'=\matr{\ast}{(0,1,0)}{(\be'_5,\ast,0)}{\ast}$ for $\be'_5\neq0$. 
\end{enumerate}
\end{lemma}
\begin{proof}
As in formula~\Ref{eq7}, denote $b=(\be_i\,|\,1\leq i\leq 8)$. 

If $\be_3=\be_4=\be_6=\be_7=0$, then we have case 1. 

If $\be_3=\be_4=0$ and $\be_5\neq0$, then by part (c) of Lemma~\ref{lemma_new} we can reduce $b$ to case 1 without changing $a$.

If $\be_3=\be_4=\be_5=0$ and $\be_6$ or $\be_7$ is non-zero,  then by part (d) of Lemma~\ref{lemma_new} we can reduce $b$ to case 2 without changing $a$. 

Let $\be_3$ or $\be_4$ be non-zero. Hence applying Remark~\ref{rem1} we can reduce $b$ to $\matr{\ast}{(0,1,0)}{(\be'_5,\be'_6,\be'_7)}{\ast}$  without changing $a$. If $\be'_5=0$, then we have case 3. Otherwise, we use   Remark~\ref{rem1} to obtain $\be'_7=0$ without changing $a$, i.e., we have case 4. The proof is completed. 
\end{proof}

\begin{proof_of}{of Theorem~\ref{theo_prop_Can4.1}} By Proposition~\ref{prop_Can3.1_new} we can assume that $a$ belongs to set (D) or ($\mathrm{K_1}$). 

Let $a$ belong to set (D).  If $a$ lies in set (E), then applying Proposition~\ref{prop_Can3.1_new} to $b$ we can assume that $b$ belongs to set (D) or ($\mathrm{K_1}$), since $\G a=a$. Assume $a$ belongs to set (F). Since the action of $\SL_3<\G$ and $\hbar\in\G$ sends octonions from set (F) to octonions from set (F), by Lemma~\ref{lemma_Claim_4.3} we can assume that $b$ belongs to one of the following sets: {\rm (D)}, {\rm (K)}, {\rm (N)}, {\rm (P)}.
Therefore, $(a,b)$ has been reduced to a pair from one of the sets from the formulation of the theorem. 

Let $a$ belong to  from set ($\mathrm{K_1}$). By Lemma~\ref{lemma_Claim_4.4} we can assume that $b$ satisfies one of conditions 1, 2, 3, 4 from Lemma~\ref{lemma_Claim_4.4}. 

\medskip
\noindent{\bf 1.} Let $b=\matr{\be_1}{(\be_2,0,0)}{(\be_5,0,0)}{\be_8}$ for $\be_i\in\FF$. 

Assume $\be_5\neq0$. There exists $u_1\in\FF$ such that $\be_5 u_1^2 + (\be_8 - \be_1)u_1 - \be_2 = 0$. Then for $g=\de_1(u_1,0,0)$ of $\G$ we have $ga=a$ and $gb=\matr{\ast}{\zero}{(\be_5,0,0)}{\ast}$, i.e., $g(a,b)$ belongs to set  ($\mathrm{K_1L^{\!\top}}$).

Assume $\be_5=0$. In case $\be_1=\be_8$ the pair $(a,b)$ belongs to set ($\mathrm{K_1L_1}$) or ($\mathrm{K_1E}$). Otherwise, for $g=\de_1(\be_2/(\be_8-\be_1),0,0)$ of $\G$ we have $g(a,b)=(a,\be_1 e_1 + \be_8 e_2)$ belongs to set  ($\mathrm{K_1F}$).

\medskip
\noindent{\bf 2.} Let $b=\matr{\be_1}{(\be_2,0,0)}{(0,1,0)}{\be_8}$ for $\be_i\in\FF$. 

Assume $\be_1\neq\be_8$. Then for $g=\de_2(0,1/(\be_1-\be_8),-\be_2)$ of $\G$ we have $ga=a$ and $gb=\matr{\be_1}{\zero}{(0,0,\be'_7)}{\be_8}$. If $\be'_7=0$, then $g(a,b)$ belongs to set ($\mathrm{K_1F}$). Otherwise, acting by  $\de_2(0,0,\be'_7/(\be_1-\be_8))$ we send $(a,gb)$ to $(a,\be_1 e_1 + \be_8 e_2)$, i.e., $(a,b)$ has been reduced to a pair from set ($\mathrm{K_1F}$). 

Assume $\be_1=\be_8$. Then for $g=\de_2(0,0,-\be_2)$ of $\G$ we have $g(a,b) = (a,\be_1 1_{\OO} + \vv_2)$ belongs to set ($\mathrm{K_1M_1^{\!\top}}$). 

\medskip
\noindent{\bf 3.} Let $b=\matr{\be_1}{(0,1,0)}{(0,\be_6,\be_7)}{\be_8}$ for $\be_i\in\FF$. 

Assume $\be_6\neq0$. There exists $v_2\in\FF$ such that $v_2^2 + (\be_1-\be_8) v_2 - \be_6 = 0$. Then for $g=\de_2(0,v_2,v_2\be_7/\be_6)$ of $\G$ we have $ga=a$ and $gb=\matr{\ast}{(0,1,0)}{\zero}{\ast}$, i.e., $g(a,b)$ belongs to set ($\mathrm{K_1M}$).

Assume $\be_6=0$. Then for $g=\de_1(-\be_7,0,0)$ of $\G$ we have $ga=a$ and $gb=\matr{\be_1}{(\ast,1,0)}{\zero}{\be_8}$. Applying Remark~\ref{rem1} we can reduce $gb$ to $\matr{\be_1}{(0,1,0)}{\zero}{\be_8}$  without changing $a$, i.e., $(a,b)$ has been reduced the pair from set ($\mathrm{K_1M}$).

\medskip
\noindent{\bf 4.} Let $b=\matr{\be_1}{(0,1,0)}{(\be_5,\be_6,0)}{\be_8}$ for $\be_i\in\FF$ with $\be_5\neq0$. 

For $g=\de_2(0,0,1/\be_5)$ of $\G$ we have $ga=a$ and $gb=\matr{\be_1}{(\ast,0,0)}{(\be_5,\be_6,\ast)}{\be_8}$. Applying Remark~\ref{rem1} we can reduce $gb$ to $\matr{\be_1}{(\ast,0,0)}{(\be_5,0,0)}{\be_8}$  without changing $a$, i.e., $b$ has been reduced to condition 1 (see above).
\end{proof_of}

\begin{theo}\label{theo_prop_Can_min} The following set is a minimal set of representatives of $\G$-orbits on $\OO^2$: 
\begin{enumerate}
\item[{\rm (EE)}] $(\al_1 1_{\OO}, \be_1 1_{\OO})$, 

\item[{\rm (E$\mathrm{\ov{F}}$)}] $\left(\al_1 1_{\OO}, \matr{\be_1}{\zero}{\zero}{\be_8}\right)$ with $\be_1<\be_8$,

\item[{\rm($\rm E K_1$)}] $\left(\al_1 1_{\OO}, \matr{\be_1}{(1,0,0)}{\zero}{\be_1}\right)$,

\item[{\rm ($\mathrm{\ov{F}}$D)}] $\left(\matr{\al_1}{\zero}{\zero}{\al_8}, \matr{\be_1}{\zero}{\zero}{\be_8}\right)$ with $\al_1<\al_8$,

\item[{\rm($\mathrm{\ov{F}}$K)}] $\left(\matr{\al_1}{\zero}{\zero}{\al_8}, \matr{\be_1}{(1,0,0)}{\zero}{\be_8}\right)$ with $\al_1<\al_8$,

\item[{\rm($\mathrm{\ov{F} K^{\!\top}}$)}] $\left(\matr{\al_1}{\zero}{\zero}{\al_8}, \matr{\be_1}{\zero}{(1,0,0)}{\be_8}\right)$ with $\al_1<\al_8$,

\item[{\rm($\mathrm{\ov{F}}$N)}] $\left(\matr{\al_1}{\zero}{\zero}{\al_8}, \matr{\be_1}{(1,0,0)}{(\be_5,0,0)}{\be_8}\right)$ with $\al_1 <\al_8$ and $\be_5\neq0$,

\item[{\rm($\mathrm{\ov{F}}$P)}] $\left(\matr{\al_1}{\zero}{\zero}{\al_8}, \matr{\be_1}{(1,0,0)}{(0,1,0)}{\be_8}\right)$ with $\al_1< \al_8$,

\item[{\rm($\rm K_1E$)}] $\left(\matr{\al_1}{(1,0,0)}{\zero}{\al_1}, \be_1 1_{\OO}\right)$,

\item[{\rm($\rm K_1F)$}] $\left(\matr{\al_1}{(1,0,0)}{\zero}{\al_1}, \matr{\be_1}{\zero}{\zero}{\be_8}\right)$ with $\be_1\neq \be_8$,

\item[{\rm($\rm K_1L_1$)}] $\left(\matr{\al_1}{(1,0,0)}{\zero}{\al_1}, \matr{\be_1}{(\be_2,0,0)}{\zero}{\be_1}\right)$ with $\be_2\neq 0$,

\item[{\rm($\rm K_1\ov{L}^{\!\top}$)}] $\left(\matr{\al_1}{(1,0,0)}{\zero}{\al_1}, \matr{\be_1}{\zero}{(\be_5,0,0)}{\be_8}\right)$ with $\be_1\leq \be_8$ and $\be_5\neq 0$,

\item[{\rm($\rm K_1\ov{M}$)}] $\left(\matr{\al_1}{(1,0,0)}{\zero}{\al_1}, \matr{\be_1}{(0,1,0)}{\zero}{\be_8}\right)$ with $\be_1\leq \be_8$,

\item[{\rm($\rm K_1M_1^{\!\top}$)}] $\left(\matr{\al_1}{(1,0,0)}{\zero}{\al_1}, \matr{\be_1}{\zero}{(0,1,0)}{\be_1}\right)$,
\end{enumerate}
where $\al_1,\al_8,\be_1,\be_2,\be_5,\be_8\in\FF$.
\end{theo}
\begin{proof} \noindent{\bf 1.} Denote by $S$ the set from the formulation of this theorem. Assume that $(a,b)\in \OO^2$. We can assume that $(a,b)$ belongs to one of the sets from the formulation of Theorem~\ref{theo_prop_Can4.1}. In particular, $a$ belongs to one of the sets: (E), (F), ($\mathrm{K_1}$).

Assume that $a$ belongs to set (E). If $b$ lies in set (F), then either $(a,b)$ or $\hbar(a,b)$ belongs to set  (E$\mathrm{\ov{F}}$). Thus, acting by $\G$ we can assume that $(a,b)$ lies in $S$.

Assume that $a$ belongs to set (F), but does not belong to set ($\mathrm{\ov{F}}$). If $(a,b)$ belongs to (FD), then $\hbar(a,b)$ belongs to set ($\mathrm{\ov{F}}$D), which lies in  $S$. If $(a,b)$ belongs to set (FK), then applying part (b) of Lemma~\ref{lemma_new} to $\hbar(a,b)$ we obtain a pair of octonions from set ($\mathrm{\ov{F} K^{\!\top}}$), which lies in $S$.  If $(a,b)$ belong to set (FN), then applying part (e) of Lemma~\ref{lemma_new} to $\hbar(a,b)$ we obtain a pair of octonions from set ($\mathrm{\ov{F}}$N), which lies in $S$. If $(a,b)$ belongs to set (FP), then applying part (f) of Lemma~\ref{lemma_new} to $\hbar(a,b)$ we obtain a pair of octonions from set ($\mathrm{\ov{F}}$P), which lies in $S$.

Let $a$ belong to set ($\rm K_1$). Assume that $(a,b)$ lies in set ($\rm K_1L^{\!\top}$). As in formula~\Ref{eq7}, denote $b=(\be_i\,|\,1\leq i\leq 8)$. If $\be_1>\be_8$, then for $g=\de_1((\be_1-\be_8)/\be_5,0,0)$ of $\G$ we have $ga=a$ and $gb=\matr{\be_8}{\zero}{(\be_5,0,0)}{\be_1}$. 
Hence, we can assume that $\be_1\leq\be_8$, i.e., we reduced set ($\rm K_1L^{\!\top}$) to set  ($\rm K_1\ov{L}^{\!\top}$). 

Assume that $(a,b)$ belongs to set ($\rm K_1M$). Denote $b=(\be_i\,|\,1\leq i\leq 8)$. If $\be_1>\be_8$, then for $g=\de_2(0,\be_8-\be_1,0)$ of $\G$ we have $ga=a$ and  $gb =\matr{\be_8}{(0,1,0)}{\zero}{\be_1}$.
Hence, we can assume that $\be_1\leq\be_8$, i.e., we reduced set ($\rm K_1M$) to set ($\rm K_1\ov{M}$). 
Therefore, $S$ is a set of representatives of $\G$-orbits on $\OO^2$.

\medskip
\noindent{\bf 2.} Let us prove the minimality of $S$. Assume that $g (a,b) = (a',b')$ for some $(a,b)$, $(a',b')$ from $S$ and some $g\in\G$. By Part 1 of Proposition~\ref{prop_Can3.1_new} we have $a=a'$. As in formula~\Ref{eq7}, we denote $a=(\al_i\,|\,1\leq i\leq 8)$, $b=(\be_i\,|\,1\leq i\leq 8)$, and $b'=(\be'_i\,|\,1\leq i\leq 8)$.

\medskip
\noindent{\bf 2.a)} Assume that $a$ belongs to set (E), i.e., $a=\al_1 1_{\OO}$. Applying Part 1 of Proposition~\ref{prop_Can3.1_new} to $b$ and $b'$, we obtain that $b=b'$.

\medskip
\noindent{\bf 2.b)} Assume that $a$ belongs to set ($\mathrm{\ov{F}}$), i.e., $a=\al_1 e_1 + \al_8 e_2$ with $\al_1<\al_8$. Since $ga=g$, Lemma~\ref{lemma_St_D} implies that $g\in\SL_3$. In particular, 
\begin{eq}\label{eq_bb}
\be_1=\be'_1 \text{ and  }\be_8=\be'_8, 
\end{eq}%
\begin{eq}\label{eq_zero}
(\be_2,\be_3,\be_4)=0 \Leftrightarrow (\be'_2,\be'_3,\be'_4)=0 \;\text{ and }\;
(\be_5,\be_6,\be_7)=0 \Leftrightarrow (\be'_5,\be'_6,\be'_7)=0. 
\end{eq}%
The pairs $(a,b)$ and $(a,b')$ belong to sets from the following list: ($\mathrm{\ov{F}}$D), ($\mathrm{\ov{F}}$K),($\mathrm{\ov{F}K^{\!\top}}$), ($\mathrm{\ov{F}}$N), ($\mathrm{\ov{F}}$P).

Conditions~\Ref{eq_zero} imply that $(a,b)$ belongs to set  ($\mathrm{\ov{F}}$D) if and only if $(a,b')$ belongs to set  ($\mathrm{\ov{F}}$D);   $(a,b)$ belongs to set  ($\mathrm{\ov{F}}$K) if and only if $(a,b')$ belongs to set ($\mathrm{\ov{F}}$K);   $(a,b)$ belongs to set  ($\mathrm{\ov{F}K^{\!\top}}$) if and only if $(a,b')$ belongs to set ($\mathrm{\ov{F}K^{\!\top}}$). In all these three cases we have $b=b'$ by equalities~\Ref{eq_bb}.

Assume that $(a,b)$ belongs to set ($\mathrm{\ov{F}}$N) and $(a,b')$ belongs to set ($\mathrm{\ov{F}}$P). Then $n(b)=\be_1\be_8 - \be_5$ is not equal to   $n(b')=\be'_1\be'_8$ by equalities~\Ref{eq_bb}; a contradiction. Therefore, $(a,b)$ and $(a,b')$ both belong to set ($\mathrm{\ov{F}}$N) or both belong to set ($\mathrm{\ov{F}}$P). Formulas~\Ref{eq_bb} together with equality $n(b)=n(b')$ imply that $b=b'$.

\medskip
\noindent{\bf 2.c)}  Assume that $a$ belongs to set  ($\mathrm{K}_1$).  Since $ga=g$,  Lemma~\ref{lemma_St_K1} implies that $g(\uu_1)=\uu_1$. The pairs $(a,b)$ and $(a,b')$ belong to sets from the following list: ($\rm K_1E$), ($\rm K_1F$), ($\rm K_1L_1$),  ($\rm K_1\ov{L}^{\!\top}$), ($\rm K_1\ov{M}$), ($\rm K_1M_1^{\!\top}$). In all these cases we have that $n(b)=\be_1\be_8$ and $n(b')=\be'_1\be'_8$. Since $\tr(b)=\tr(b')$, we can see that 
\begin{eq}\label{eq_2c}
\be_1=\be'_1,\; \be_8=\be'_8 \text{ or } \be_1=\be'_8,\; \be_8=\be'_1.
\end{eq}

Since $\G 1_{\OO}=1_{\OO}$, we have that $(a,b)$ belongs to set  ($\rm K_1E$) if and only if $(a,b')$ belongs to set ($\rm K_1E$); and in this case $b=b'$.

Assume that $(a,b)$ belongs to set ($\rm K_1F$).  Since $ga=g$,  Lemma~\ref{lemma_St_K1} implies that $g b=\matr{\be_1}{(\ast,0,0)}{(0,\ast,\ast)}{\be_8}$.   Therefore, $(a,b')$ does not belong to sets ($\rm K_1\ov{M}$) and ($\rm K_1\ov{L}^{\!\top}$). Since $\be_1\neq\be_8$,  $(a,b')$ does not belong to sets ($\rm K_1L_1$) and  ($\rm K_1M_1^{\!\top}$). Therefore,  $(a,b')$ also belongs to set  ($\rm K_1F$) and $b=b'$. 

Assume that $(a,b)$ belongs to set ($\rm K_1L_1$).  Since $g(\uu_1)=\uu_1$, we have $g(b)=g(\be_1 1_{\OO} + \be_2 \uu_1) = b$, i.e., $b=b'$.  

Now we have that pairs $(a,b)$ and $(a,b')$ belong to sets from the following list:  ($\rm K_1\ov{L}^{\!\top}$), ($\rm K_1\ov{M}$), ($\rm K_1M_1^{\!\top}$). 

Assume that $(a,b)$ belongs to set ($\rm K_1\ov{L}^{\!\top}$). In this case we have that  $\tr(ab)=\be_5 + \al_1(\be_1+\be_8)$. If  $(a,b')$ belongs to sets ($\rm K_1\ov{M}$) or ($\rm K_1M_1^{\!\top}$), then $\tr(ab')= \al_1(\be'_1+\be'_8)=\al_1(\be_1+\be_8)$ by equalities~\Ref{eq_2c}; a contradiction to  $\tr(ab)=\tr(ab')$, since $\be_5\neq0$. Therefore, $(a,b')$ also belongs to set  ($\rm K_1\ov{L}^{\!\top}$). Equalities~\Ref{eq_2c} imply that $\be_1=\be'_1$ and $\be_8=\be'_8$. Since $\tr(ab)=\tr(ab')=\be'_5 + \al_1(\be'_1+\be'_8)$, we obtain that $\be_5=\be'_5$. Therefore, $b=b'$.  

Assume that $(a,b)$ belongs to set ($\rm K_1M_1^{\!\top}$). For each $b_1,b_2\in\OO$ denote $f(b_1,b_2)=b_1 b_2 - \al_1 b_2$. Note that $g(f(a,b))=f(ga,gb)=f(a,b')$. Since $f(a,b)=\be_1\uu_1$ and $g\uu_1=\uu_1$, we have that $f(a,b')=\be_1\uu_1$. If  $(a,b')$ belongs to set ($\rm K_1\ov{M}$), then $f(a,b')=\be'_8\uu_1 + \vv_3$; a contradiction. Therefore,  $(a,b')$ also belongs to set ($\rm K_1M_1^{\!\top}$).  Equalities~\Ref{eq_2c} imply that $b=b'$. 

Therefore, it remains to deal with the case when $(a,b)$ and $(a,b')$ belong to set ($\rm K_1\ov{M}$). Since $\be_1\leq\be_8$ and $\be'_1\leq \be'_8$, equalities~\Ref{eq_2c} imply that $b=b'$. The proof is concluded. 
\end{proof}

\section*{Acknowledgement} 

We are very grateful to the anonymous referee for helpful comments.


\bigskip

\bibliographystyle{siam}
\bibliography{main}

\end{document}